\documentclass{amsart}
\usepackage{a4}

\usepackage{color}
\usepackage{amsthm}
\usepackage{amsmath}
\usepackage{amssymb} 
\usepackage{paralist}
\usepackage{cite}

\theoremstyle{definition}
\newtheorem{Definition}{Definition}[section]
\newtheorem{Theorem}[Definition]{Theorem}
\newtheorem{Corollary}[Definition]{Corollary}
\newtheorem{Lemma}[Definition]{Lemma}
\newtheorem{Remark}[Definition]{Remark}
\newtheorem{Proposition}[Definition]{Proposition}

\newtheorem{Construction}[Definition]{Construction}

\title{A multisymplectic manifold not covered by Darboux charts}
\author{Leonid Ryvkin}
\address{Leonid Ryvkin, Fakult\"at f\"ur Mathematik, Ruhr-Universit\"at Bochum, Universit\"atsstr. 150, 44801 Bochum, Germany}
\email{Leonid.Ryvkin@ruhr-uni-bochum.de}

\begin{document}

\maketitle

\begin{abstract}
The Darboux theorem in symplectic geometry implies that any two points in a connected symplectic manifold have neighbourhoods symplectomorphic to each other. The impossibility of such a theorem in the more general multisymplectic framework appears to be, at least, folkloristic, but no explicit counterexample seems to exist in the literature.
 
In this note we provide such an example by constructing multisymplectic three-forms on the connected manifold $\mathbb R^6$, which do not even have constant linear type and therefore can not allow for an atlas consisting of ``Darboux charts''.
\end{abstract}

\section*{Introduction}
One of the most striking facts of symplectic geometry is the Darboux theorem, saying that given a symplectic manifold $(M,\omega)$ of real dimension $2n$ and a point $p\in M$, there exists a chart $(U,\phi)$ of $M$ near $p$, such that $\tilde\omega=(\phi^{-1})^*(\omega|_U)$ is a two-form with constant coefficients on the open set $\phi(U)\subset \mathbb R^{2n}$. In fact, $\phi$ can be chosen in a way, such that $\tilde \omega=\sum_{i=1}^ne^i\wedge e^{n+1}$, where $\{e^1,...,e^{2n}\}$ is the standard basis of $(\mathbb R^{2n})^*$.

Going to multisymplectic forms, i.e. $k$-forms $\alpha$ on a manifold $M$, which are non-degenerate in the sense that the map $TM\to \Lambda^{k-1}T^*M,~v\mapsto \iota_v\omega$ is injective, one can ask if the following ``naive'' generalization of the Darboux theorem holds:

\begin{quote}\emph{
Does every $m$-dimensional multisymplectic manifold $(M,\alpha)$ allow for an atlas consisting of ``Darboux charts'' $(U,\phi)$, such that $(\phi^{-1})^*(\omega|_U)$ has constant coefficients on $\phi(U)\subset \mathbb R^m$?}
\end{quote}

The answer to this question is no, which is certainly known in the field (cf., e.g., \cite[Remark 31]{MR2253159}), but we could not pinpoint a reference explictly exhibiting a counterexample. In this short note we construct for degree $k=3$ a counterexample to the above question on $\mathbb R^6$.\\

The first section recapitulates the fundamental facts about the action of $GL(V)$ on the non-degenerate elements of $\Lambda^3V^*$ for a six-dimensional real vector space $V$, whereas in the second section we construct a class of multisymplectic three-forms on $\mathbb R^6$ not allowing an atlas of Darboux charts.\\

If the necessary condition of locally constant linear type and additional further constraints are imposed on a multisymplectic manifold, restricted Darboux-like theorems can be found in the literature (cf., e.g., \cite{MR1694063, MR962194, MR801210,MR3011894} and the internet source \cite{126197}).\\

{\bf  Acknowledgements.}  The author thanks Tilmann Wurzbacher for posing the problem considered in this article and for useful discussions.
\section{Equivalence classes of non-degenerate three-forms in dimension six}
On a finite-dimensional real vector space $V$ an alternating form $\alpha\in \Lambda^kV^*$ is called \emph{non-degenerate} if the map $V\to \Lambda^{k-1}V^*, v\mapsto \iota_v\alpha$ is injective. Then $(V,\alpha)$ is called a \emph{$(k{-}1)$-plectic} (or \emph{multisymplectic}) vector space. Two $(k{-}1)$-plectic vector spaces $(V,\alpha)$, $(\tilde V,\tilde \alpha)$ of the same dimension are said to have \emph{the same linear type}, if there exists a linear isomorphism $\psi:V\to \tilde V$, such that $\psi^*\tilde \alpha=\alpha$. Fixing any linear isomorphism $\psi:V \to \tilde V$, this is the case if and only if $\alpha$ and $\psi^*\tilde \alpha$ lie in the same orbit of the $GL(V)$-action on $\Lambda^kV^*$.\\

Obviously non-degenerate 1-forms  can only occur on vector spaces of dimensions 0 and 1, so the case $k=1$ will be excluded in the sequel.\\

 In the case $k=2$ a non-degenerate form $\alpha$ exists, if and only if $dim_\mathbb RV$ is even. Moreover, by the symplectic basis theorem, for any basis $\{e^1,...,e^{2n}\}$ of $V^*$, there exists a linear isomorphism $L\in GL(V)$, such that $L^*\alpha=\sum_{j=1}^n e^j\wedge e^{j+n}$. Especially all non-degenerate two-forms form a single $GL(V)$-orbit inside $\Lambda^2V^*$, i.e. for a fixed even-dimensional vector space there exists only one linear type of symplectic forms.\\

In stark contrast to the symplectic case, for higher degree forms there can be many non-degenerate $GL(V)$-orbits. For $k=3$ multiple non-degenerate orbits exist, if and only if the dimension of $V$ is greater or equal to six (cf. \cite{MR0286119, MR691457}). Let us recall the six-dimensional case in more detail.

\begin{Construction}
Let $V$ be a six-dimensional real vector space and $\alpha\in \Lambda^3V^*$ non-degenerate. Regard the map $J_\alpha:V\to \Lambda^5V^*, J_\alpha(v)=(\iota_v\alpha)\wedge \alpha$. As $\Lambda^5V^*$ is naturally $GL(V)$-equivariantly isomorphic to $V\otimes \Lambda^6V^*$, we can interpret $J_\alpha$ as an element of $End(V)\otimes \Lambda^6V^*$ and consequently define $J_\alpha^2\in End(V)\otimes (\Lambda^6V^*)^{\otimes 2}$ and $trace(J_\alpha^2)\in (\Lambda^6V^*)^{\otimes 2}$.
\end{Construction}

\begin{Theorem}\cite{MR2253159,MR1863733}\label{normalf6}
Let $V$ be a six-dimensional real vector space, $\Omega\in\Lambda^6V^*\backslash\{0\}$ a volume form and $\{e^1,...,e^6\}$ an ordered basis of $V^*$. Let $\alpha \in \Lambda^3V^*$ be non-degenerate. Then there is a unique scalar $\lambda_\alpha\in \mathbb R$, such that $trace(J_\alpha^2)=\lambda_\alpha\cdot(\Omega\otimes \Omega)\in(\Lambda^6V^*)^{\otimes 2}$. The non-vanishing and sign of $\lambda_\alpha$ do  not depend on the choice of volume form and we have:

\begin{enumerate}[(i)]
	\item there exists $g\in GL(V)$ such that $g^*(\alpha)=\alpha_{(i)}$ if and only if $\lambda_\alpha>0$, 
	\item there exists $g\in GL(V)$ such that $g^*(\alpha)=\alpha_{(ii)}$ if and only if $\lambda_\alpha<0$, 
	\item there exists $g\in GL(V)$ such that $g^*(\alpha)=\alpha_{(iii)}$ if and only if $\lambda_\alpha=0$, 
\end{enumerate}
where 
\begin{align*}
&\alpha_{(i)}=e^1\wedge e^2\wedge e^3+ e^4\wedge e^5\wedge e^6\\
&\alpha_{(ii)}=e^1\wedge e^3\wedge e^5 - e^1\wedge e^4\wedge e^6 - e^2\wedge e^3\wedge e^6-e^2\wedge e^4\wedge e^5\\
&\alpha_{(iii)}=e^1\wedge e^5\wedge e^6 - e^2\wedge e^4\wedge e^6 + e^3\wedge e^4\wedge e^5.
\end{align*}
\end{Theorem}

\section{A 2-plectic structure on $\mathbb R^6$ with varying linear type}

\begin{Definition} A smooth manifold $M$ equipped with a closed differential form $\omega\in \Omega^k_{cl}(M)$ is called \emph{multisymplectic of degree $k$} or \emph{$(k{-}1)$-plectic} if $\omega_p\in \Lambda^kT^*_pM$ is non-degenerate for all $p\in M$.
\end{Definition}

For the naive generalization of the Darboux theorem from symplectic geometry (cf. \cite{MR0182927}) to hold in multisymplectic geometry, each point in any multisymplectic manifold $(M,\omega)$ must necessarily have a neighbourhood $U$, such that $\omega_p$ has the same linear type for all $p$ in $U$.\\

We will now construct a counter-example to the above by exhibiting a multisymplectic manifold $(M,\omega)$ with variable linear types in the ``smallest'' possible situation, i.e., $dim_\mathbb RM=6$ and $k=3$.

\begin{Lemma}
Writing $dx^{ijk}$ for $dx^i\wedge dx^j\wedge dx^k$, the form
\[\alpha=\alpha^t= dx^{135} -dx^{146}-dx^{236} + t \cdot dx^{245}\in \Lambda^3{\mathbb R ^6}^*\]
is non-degenerate for all $t \in \mathbb R$ and 
\[ trace(J_{\alpha^t}^2)=24\cdot t\cdot (dvol)^{\otimes 2}, \]
where $dvol= dx^1\wedge dx^2\wedge dx^3\wedge dx^4\wedge dx^5\wedge dx^6$.
\end {Lemma}
\begin {proof}
We write $e_i$ for the standard basis vectors of $\mathbb R^6$ and $dx^{ij}$ instead of $dx^i\wedge dx^j$. We have

\begin{minipage}{0.5\textwidth}
\begin{align*}
&\iota_{e_1}\alpha=dx^{35} - dx^{46}\\
&\iota_{e_2}\alpha=-dx^{36} +t dx^{45}\\
&\iota_{e_3}\alpha=-dx^{15} + dx^{26}\\
\end{align*}
\end{minipage}\begin{minipage}{0.5\textwidth}
\begin{align*}
&\iota_{e_4}\alpha=dx^{16} - t dx^{25}\\
&\iota_{e_5}\alpha=dx^{13} +t dx^{24}\\
&\iota_{e_6}\alpha=-dx^{14} - dx^{23}.\\
\end{align*}
\end{minipage}

As no pair of indices above appears twice, no non-trivial linear combination of the above 2-forms will yield zero, i.e.  $\alpha=\alpha^{t}$ is non-degenerate for all $t $. Next we calculate $J_\alpha$:

\begin{minipage}{0.4\textwidth}
\begin{align*}
&J_\alpha (e_1)=-dx^{35146}-dx^{46135}=2dx^{13456}=-2\iota_{e_2}dvol\\
&J_\alpha (e_2)=-2t \iota_{e_1}dvol\\
&J_\alpha (e_3)=-2\iota_{e_4}dvol\\
\end{align*}
\end{minipage}
\begin{minipage}{0.35\textwidth}
\begin{align*}
&J_\alpha (e_4)=-2 t\iota_{e_3}dvol\\
&J_\alpha (e_5)=2t\iota_{e_6}dvol\\
&J_\alpha (e_6)=2\iota_{e_5}dvol.\\
\end{align*}
\end{minipage}

Applying the canonical isomorphism $\Lambda^5\mathbb R^6\cong \mathbb R^6\otimes \Lambda^6(\mathbb R^6)^*$  we get
\[
J_\alpha=\begin{pmatrix}
0&-2t &0&0&0&0\\
-2&0&0&0&0&0\\
0&0&0&-2t &0&0\\
0&0&-2&0&0&0\\
0&0&0&0&0&2\\
0&0&0&0&2t&0\\
\end{pmatrix}\otimes dvol\in End(\mathbb R^6)\otimes \Lambda^6(\mathbb R^6).
\]
Consequently $(J_\alpha)^2=4t\cdot id_{\mathbb R^6}\otimes (dvol)^{\otimes 2}$, which finishes the proof.
\end {proof}

The above lemma shows, that $\{\alpha^t|t\in\mathbb R\}$ contains elements of all the linear types discussed in Theorem \ref{normalf6}. We now use this fact to construct a multisymplectic structure on $\mathbb R^6$, which has different linear types at different points.

\begin{Proposition} \label{prop}
Let $f:\mathbb R^6\to \mathbb R$ be a smooth function. The manifold $\mathbb R^6$ equipped with the form 
\[\omega^f= dx^{135} -dx^{146}-dx^{236} + f(x)\cdot dx^{245} \]
 is multisymplectic of degree 3, if $f$ only depends on $x_2,x_4$ and $x_5$. For $p$ in $\mathbb R^6$, the form $\omega_p^f$ has the linear type of $\alpha_{(i)}$ if $f(p)>0$, $\alpha_{(ii)}$ if $f(p)<0$ and $\alpha_{(iii)}$ if $f(p)=0$.
\end{Proposition}
 \enlargethispage{10px}
\begin{Corollary}
The multisymplectic manifold $(\mathbb R^6,\omega^f)$, with $f(x)=x_2$ does not admit any neighbourhood of $0$, on which the linear type of $(T_p\mathbb R^6,\omega_p^f)$ is constant. Notably, there is no chart $(U,\phi)$ of $\mathbb R^6$ near 0, such that $(\phi^{-1})^*(\omega^f)$ is a three-form with constant coefficients on $\phi(U)\subset \mathbb R^6$. Furthermore the group 
\[\text{Diff}_{\omega^f}(\mathbb R^6)=\{\Phi:\mathbb R^6\to \mathbb R^6~|~ \Phi \text{ is a diffeomorphism and }\Phi^*(\omega^f)=\omega^f \}\]
does not act transitively on $\mathbb R^6$.
\end{Corollary}
\begin{proof}
The first claim follows directly from Proposition \ref{prop}. Assuming that $\text{Diff}_{\omega^f}(\mathbb R^6)$ acts transitively, it especially has to contain an element $\Phi$ mapping $0$ to $e_2=(0,1,0,0,0,0)$. Then $(D_0\Phi)^*(\omega^f_{e_2})=\omega^f_0$, i.e.,  $\omega^f_0$ must have the same linear type as $\omega^f_{e_2}$, contradicting the choice of $f$.
\end{proof}

\begin{Remark}
In terms of the introduction, the multisymplectic manifold $(\mathbb R^6, \omega^f)$ does not allow for a Darboux chart near 0.
\end{Remark}

\begin{Remark}
Compact examples can be constructed by choosing $f$ to be periodic and taking the quotient with respect to a group action. For instance, we can set $f(x)=sin(2\pi x_2)$. Then $\omega^f$ descends to a multisymplectic form on $T^6=\mathbb R^6/\mathbb Z^6$, the 6-dimensional compact torus, which still has different linear types in different points.
\end{Remark}

\bibliographystyle{habbrv}
\bibliography{document}

\begin{thebibliography}{10}

\bibitem{126197}
R.~L. Bryant.
\newblock Darboux like theorem for non-degenerate 3-forms in 6-manifolds.
\newblock MathOverflow, http://mathoverflow.net/q/126197 (version: 2013-04-02).

\bibitem{MR2253159}
R.~L. Bryant.
\newblock On the geometry of almost complex 6-manifolds.
\newblock {\em Asian J. Math.}, 10(3):561--605, 2006.

\bibitem{MR1694063}
F.~Cantrijn, A.~Ibort, and M.~de~Le{\'o}n.
\newblock On the geometry of multisymplectic manifolds.
\newblock {\em J. Austral. Math. Soc. Ser. A}, 66(3):303--330, 1999.

\bibitem{MR691457}
D.~{\v{Z}}. Djokovi{\'c}.
\newblock Classification of trivectors of an eight-dimensional real vector
  space.
\newblock {\em Linear and Multilinear Algebra}, 13(1):3--39, 1983.

\bibitem{MR3011894}
A.~Echeverr{\'{\i}}a-Enr{\'{\i}}quez, A.~Ibort, M.~C. Mu{\~n}oz-Lecanda, and
  N.~Rom{\'a}n-Roy.
\newblock Invariant forms and automorphisms of locally homogeneous
  multisymplectic manifolds.
\newblock {\em J. Geom. Mech.}, 4(4):397--419, 2012.

\bibitem{MR1863733}
N.~Hitchin.
\newblock The geometry of three-forms in six dimensions.
\newblock {\em J. Differential Geom.}, 55(3):547--576, 2000.

\bibitem{MR962194}
G.~Martin.
\newblock A {D}arboux theorem for multi-symplectic manifolds.
\newblock {\em Lett. Math. Phys.}, 16(2):133--138, 1988.

\bibitem{MR0286119}
J.~Martinet.
\newblock Sur les singularit\'es des formes diff\'erentielles.
\newblock {\em Ann. Inst. Fourier (Grenoble)}, 20(fasc. 1):95--178, 1970.

\bibitem{MR0182927}
J.~Moser.
\newblock On the volume elements on a manifold.
\newblock {\em Trans. Amer. Math. Soc.}, 120:286--294, 1965.

\bibitem{MR801210}
J.~F. Turiel.
\newblock {\em Classification locale des 3-formes ferm\'ees
  infinit\'esimalement transitives \`a cinq variables}, volume~30 of {\em
  Cahiers Math\'ematiques Montpellier [Montpellier Mathematical Reports]}.
\newblock Universit\'e des Sciences et Techniques du Languedoc, U.E.R. de
  Math\'ematiques, Montpellier, 1984.

\end{thebibliography}
\end{document}